\documentclass[a4paper, 12 pt]{article} 

\usepackage{amsmath} % assumes amsmath package installed
\usepackage{amssymb}  % assumes amsmath package installed
\usepackage{amsthm}
\usepackage{graphicx}
\usepackage{tikz-cd}

\newtheorem{definition}{Definition}
\newtheorem{theorem}{Theorem}
\newtheorem{proposition}{Proposition}

\def\b{\mathbf{b}}

\def\v0{\mathbf{0}}

\def\v0{\mathbf{0}}

\newcommand{\R}{{\mathbf R}}

\newcommand{\cC}{{\mathcal C}}

\newcommand{\spanek}{\mathrm{span}}

\title{\LARGE \bf
The singular set and aspects of global control dynamics
}

\author{Efthimios Kappos}

\begin{document}

\maketitle
\thispagestyle{empty}
\pagestyle{empty}

%%%%%%%%%%%%%%%%%%%%%%%%%%%%%%%%%%%%%%%%%%%%%%%%%%%%%%%%%%%%%%%%%%%%%%%%%%%%%%%%
\begin{abstract}

The global approach to control systems which we have been pursuing in other work favours the study of dynamics achievable through control. It employs certain globally defined geometric objects and attempts to describe them in the general case. In this work, we define the \emph{singular set} and examine some of its general properties. We then briefly examine its role in the global design of control dynamics.

\end{abstract}

%%%%%%%%%%%%%%%%%%%%%%%%%%%%%%%%%%%%%%%%%%%%%%%%%%%%%%%%%%%%%%%%%%%%%%%%%%%%%%%%
\section{Introduction}

Linear control design is by its nature local; hence the main interest is in obtaining a single stable equilibrium for the resulting control dynamics. In global control, the dynamics considered can include a number of equilibrium points of different stability types (index) and perhaps other, more complicated invariant sets, such as limit cycles.
It is therefore relevant to consider the set of all points which can be turned into equilibria through some choice of control. Choosing the term control indicatrix at a given state for the subset of the tangent space containing all possible control vectors, the set we just described is the subset of the state space where the control indicatrix contains the zero vector. It is the singular set of this work.
As a simple example, for a single-input linear control system in its controller-canonical form, the singular set is a line (one-dimensional subspace) in the direction of the first state component.

We shall give results about the generic dimension of the singular set, distinguishing between the control-affine case and the more general context of a control fibration (details later).
Since the state spaces we consider are manifolds, we must allow for singularities in the control distribution, in other words for subsets where the rank of the the span of the control vector fields drops. 
The special category of drift-free control systems is not our main consideration since, in this case, and assuming a control indicatrix containing a neighbourhood of zero, the singular set is the whole of the state space manifold.

In all the other cases, the dimension of the singular set is generically the same as the dimension of control. This is convenient for the global control methodology we have in mind, which uses submanifolds transverse to the control distribution (or fibration) and of complementary dimension. On such manifolds, we have natural dynamics defined (independent of control), which we have called control-transverse dynamics in~\cite{ek2}. Taking advantage of the transversality of the control, such a manifold can then be made invariant for the control dynamics and, moreover, we can fully control the dynamics in the transverse directions. 
If the control-transverse manifold is also transverse to the singular set, we obtain isolated equilibria, whose stability type we have some control over by the choice of the geometric object, the transverse manifold.
We indicate the main lines of this global control design method in the last section.

\section{The singular set in some common control settings}

In this work, the \textbf{singular set} is defined as the subset of state space of points where we have available the zero vector. We shall denote it by $\Sigma$. By way of motivation, we first examine the form it takes in each of the three cases we consider and describe the \emph{expected} dimension and geometry. More precise results will be stated and proved in the next section.

\subsection{Linear control systems}\label{lcs1}

For a linear control system
\[ \dot x = Ax + Bu , \quad x \in \R^n , \quad u \in \R^m , \]
writing $W^m = \spanek ( B) = \spanek ( b_1 , \ldots , \b_m )$ for the constant control distribution (assumed of full rank $m$ of course), the singular set is defined as
\[ \Sigma = \{ x \in \R^m : \, Ax \in \spanek (B) \} . \]
Since the $b_i$ are linearly independent, let us take complete to a basis of $\R^n$, $( e_1 , \ldots , e_{n-m}, b_1 , \ldots b_m )$ and write $\R^n = U^{n-m} \oplus W^m$ for the resulting direct sum decomposition. If $x_1 , x_2$ are the corresponding components of the state and we partition the matrix $A$ in blocks in the same way,
\[ A = \left[ \begin{array}{cc} A_{11} & A_{12} \\ A_{21} & A_{22} \end{array} \right] , \]
we see that $x \in \Sigma$ if and only if 
\[ A_{11} x_1 + A_{12} x_2 = 0 . \]
Generically, we expect that this system of $(n-m)$ equations in $n$ variables will have as solution an $m$-dimensional subspace, which is the singular set in this case.

If the system is controllable and is given in controller-canonical form, this is seen very easily. We do the single-input case, the general Brunovsky canonical form being similar.
Since the state matrix is
\[ A = \left[ \begin{array}{rrrrr} 0 & 1 & \cdots & \cdots & 0 \\ 0 & 0 & 1 & \cdots & 0 \\ \vdots & \vdots &\vdots &\vdots & \vdots \\ 0 & 0 & \cdots & 0 & 1 \\
- a_1 & - a_2 & \cdots & \cdots & - a_n  \end{array} \right] \]
and $b = e_n$, the singular set is the line $x_2 = 0 , \ldots , x_n =0$, i.e.\ the $x_1$ axis. In the general case, it is an $m$-dimensional subspace.

\subsection{Control-affine systems in $\R^n$}

We are interested in control-affine systems in $\R^n$
\[ \dot x = f(x) + \sum_{i=1}^m u_i g_i (x) \]
where the "drift" vector field is general (not the zero vector field).
Assuming unlimited control action, the control indicatrix at each point is an affine subspace of dimension equal to $\dim \spanek ( g_1 (x) , \ldots , g_m (x) )$. We expect this dimension to be $m$ on an open and dense subset of state space.

The singular set is again
\[ \Sigma = \{ x \in \R^n : f(x) \in \spanek ( g_1 (x ) ,  \ldots , g_m (x) ) \} . \]
Reasoning as above, we can make a local argument: near a point where the rank is maximal, consider a basis of the form $( e_1 , \ldots , e_{n-m} , g_1 , \ldots , g_m )$ (all are functions of the state, defined in some neighbourhood.) Decomposing the state in the new variables into $x_1 \in \spanek ( e_1 , \ldots , e_{n-m} )$ and $x_2 \in \spanek ( g_1 , \ldots , g_m )$, the control system takes the form
\[ \left[ \begin{array}{c} \dot x_1 \\ \dot x_2 \end{array} \right] = \left[ \begin{array}{c} f_1 ( x_1 , x_2 ) \\ f_2 ( x_1 , x_2 ) \end{array} \right] + \left[ \begin{array}{c} 0 \\ I_m \end{array} \right] u . \]
Hence $x \in \Sigma$ exactly when $f_1 ( x_1 , x_2 ) = 0$, again a (nonlinear) system of $(n-m)$ equations in $n$ variables, and we expect that the singular set has dimension $m$.
 
As an important special case, consider the form assumed in the back-stepping method, the strict-feedback form:
\begin{align*}
 \dot x_1 & = f_1 ( x_1 ) + g_1 (x_1) x_2 \\
 \dot x_2 & = f_2 (x_1, x_2) + g_2 ( x_1 , x_2 ) x_3 \\
 \cdots & \cdots \\
 \dot x_{n-1} & = f_{n-1} ( x_1 , \ldots , x_{n-1} ) + g_{n-1} ( x_1 , \ldots , x_{n-1} ) x_n \\
 \dot x_n & = f_n (x) + g(x) u 
\end{align*}
This is meant to be a local expression, and the assumption made, that the $g_i$ are non-zero, means that the singular set is simply the solution set of the system of $(n-1)$ equations
\begin{align*}
 0 &= f_1 ( x_1 ) + g_1 (x_1) x_2 \\
 0 & = f_2 (x_1, x_2) + g_2 ( x_1 , x_2 ) x_3 \\
 \cdots & \cdots \\
 0 & = f_{n-1} ( x_1 , \ldots , x_{n-1} ) + g_{n-1} ( x_1 , \ldots , x_{n-1} ) x_n ,
\end{align*}
which, by the implicit function theorem (given the assumptions made), is (locally) a line (a one-dimensional manifold, in fact a graph.) This structure is exploited in the back-stepping methodology, which is a Lyapunov function-based control design (see, e.g..~\cite{khalil}, section 14.3.)
% In fact, the assumptions are so strong that we have essentially a linear system nearby (exact linerzn)

\subsection{Control fibrations}\label{fibr-intro}

We now change gear and define a general control system on a smooth state space manifold $M^n$, using the notion of a \emph{fibration}. The use of the notion of fibration is not new in control theory (see for example~\cite{coron1}), though here we try and make the assumptions as realistic as possible, and also suitable for application of the transversality results to follow.

\begin{definition}
A subset $\cC$ of the tangent bundle $TM$ is a \textbf{control fibration} if 
\begin{equation*}
\begin{tikzcd}
\cC \arrow[r, "i"] \arrow[dr,"p"] & TM \arrow{d}{\pi} \\
 &  M 
\end{tikzcd}
\end{equation*}
with $i$ an inclusion map such that $p = \pi \circ i$ is onto and $( \cC , p, M)$ is a weak (or Serre) fibration. Each fibre is assumed to be a compact, convex subset of the corresponding tangent bundle.
\end{definition}
The choice of this definition is in recognition of the fact that we never have, in practice, unlimited control action and, moreover, by using generalized controls, we can assume that we have available any direction in the convex span of any finite set of control directions. For the purposes of this paper, it allows us to obtain results about genericity which exploit the compactness in a crucial way. Having a convex fibre does not, of course, imply that the fibration is trivial: take as a simple example the usual M\"{o}bius band as a fibre bundle with fibre the unit interval or, even more simply, the unit disk bundle of the sphere, with fibre the vectors of length less than one, with respect to some Riemannian metric.

\section{Genericity results}

\subsection{Linear control systems}

\begin{proposition}
For linear control systems
\[ \dot x = A x+ Bu , \quad x \in \R^n , \quad u \in \R^m , \; m \le n ,  \]
provided the matrix $B$ is of full rank $m$, the set of all state matrices $A$ such that the singular set is an $m$-dimensional subspace is open and dense in the space of all square matrices.
\end{proposition}

The proof is a familiar elementary argument.

\begin{proof}
As explained above, the singular set is the set of solutions of a system of $(n-m)$ linear equations in $n$ variables. Provided the $(n-m) \times n$ matrix $[ A_{11} A_{12} ]$ is of full rank, the solution set is an $m$-dimensional subspace. But this happens provided at least one of the $(n-m) \times (n-m)$ sub-matrices has non-zero determinant, an algebraic inequality in its elements. The set of such matrices is thus a union of open, dense subsets.
\end{proof}

\subsection{Control-affine systems on a manifold}

The obvious way of generalizing a control-affine system in a vector space to one in a manifold is to suppose given $n+1$ smooth vector fields on a smooth manifold $M^n$. The first is to play the role of the "drift" and the remaining $n$ vector fields span a distribution $D \subset TM$, which is smooth by construction, but may not have constant rank (see~\cite{isidori}). This leads to the posing of problems such as that of feedback equivalence (see~\cite{bronik1}), which have been well developed within nonlinear control theory. At its core, this is just a manifestation of affine geometry: there is no natural selection of a drift vector field, since any choice of feedback control changes it. This is just saying that an affine subspace of a vector space does not have a natural affine basis.

There is a simple algebraic way of making everything invariant: we consider instead of the control distribution, the quotient bundle $TM/D$. 
\emph{An affine control system (with unbounded control) is then simply a section of this bundle.} 
When the rank of $D$ is constant and equal to to $m$, the quotient bundle is a vector bundle of fibre dimension $(n-m)$.
For any vector bundle $(E, \pi , M)$, where $\pi : E \to M$ is the projection, and the fibre is isomorphic to $\R^k$, there is the obvious zero section $z: M \to E$, $z (p) = (p, 0) \in E_p$, which allows us to identify a copy of $M$ in the total space $E$ with the image of the zero section. We denote it by $Z$. It is a submanifold of $E$ of dimension $n = \dim M$.
The smooth sections of any vector bundle $E \to B$ are denoted by $\Gamma (E)$. Thus, a control-affine system is an element of $\Gamma (TM/D)$. \emph{The zeros of such a section correspond to points of the singular set.}

The rank of the control distribution $D$ is not constant, in fact \emph{cannot} be constant on a general manifold, for topological reasons at least (any vector field on an even sphere must have zeros, for example.) 
The singularities of such distributions, generated by $m$ vector fields, have been studied from a local point of view (\cite{bronik1}, \cite{mormul1}, \cite{mormul2}, \cite{mormul-rouss}), motivated by control theory considerations. Local normal forms have been obtained, though, very quickly, as we increase (co)dimension, moduli appear.
On the other hand, it is a classical topic in singularity theory to consider the stratification of the jet space associated with a smooth map between manifolds according to rank of the derivative of a map (see for example \cite{via-sdm} or \cite{gol-guill}.)
Since we are interested in global aspects, and since we are only considering control-affine systems (with general non-zero drift term), we only need the global genericity results of singularity theory, adapted to the context of the distribution $D$.
This means that we use local charts to identify the $m$-tuple of vector fields with the $n \times m$ matrix they define. As in the case of the jet space $J^1 (M,N)$, we have first a stratification of the space of $n\times m$ matrices according to corank (\cite{gol-guill}, Prop.5.3, p.60) and then by transversality we deduce the existence of submanifolds of $M$ where the control distribution has this corank. 
The lowest stratum, which we shall call $M_0$, consists of points where the matrix has full rank $m$. The higher strata have higher codimension (so lower dimension) and are described in the references sited. The submanifold $M_0$ is an open and dense subset of $M$ (it has zero codimension in $M$.)
Throughout, we consider the Whitney topology on the space of smooth maps between manifolds (see~\cite{gol-guill} for details.)

\begin{theorem}
For a dense subset of the set of $m$-tuples of vector fields, the subset $M_0$ of $M$ where the rank of $D$ is $m$ is the open, dense submanifold of $M$ and the singular set is an $m$-dimensional submanifold of $M_0$. 
\end{theorem}

\begin{proof}
The first part follows from the singularity theory considerations above.
We shall obtain the singular set as the intersection of the section of the quotient bundle with the $n$-dimensional submanifold $Z$, the zero section of $TM/D$. In a suitable local chart $(U, \phi)$, the bundle is a product bundle, $(TM/D)|_{U_i} \simeq U \times \R^{n-m}$. Cover $M_0$ by a countable collection of charts $(U_i, \phi_i)$, with the closure of each $U_i$ compact. By transversality theory, the set of maps $s : U_i \to (TM/D)|_{U_i}$ which are transverse to the $n$-dimensional submanifold $Z \cap (TM/D)|_{U_i}$ is open and dense in the relevant Whitney topology. But this means that the inverse image of the intersection is a submanifold of $U_i$ of the same codimension as the codimension of $Z$ in $(TM/D)|_{U_i} \simeq U_i \times \R^{n-m}$. This codimension is clearly $n-m$ so the submanifold we want is of dimension $n - (n-m) =m$. This is the piece of the singular set in $U_i$. Patching these together and using the fact that we have a countable cover, we conclude that the singular set is an $m$-dimensional submanifold for a dense subset of the set of sections of the quotient bundle, which is the same as the set of control-affine systems.
\end{proof}

\subsection{Control fibrations}

In section~\ref{fibr-intro}, we gave a definition of a general control system with bounded control action and a convexity assumption, motivated by considerations of generalized controls and we shall use this definition without further comment. We now make a further assumption on the control fibration $\cC \to M$, to reflect the fact that the control indicatrix comes usually as a subset of some affine subspace of the tangent space.

First, then, we assume that a (singular) distribution $D \subset TM$ is given, of generic dimension $m$. We have, again, an open and dense submanifold $M_0$ on which the distribution is regular, of rank $m$.
The control fibration is now further restricted to having each fibre a compact, convex subset of some element of $T_p M_0 / D_p$, in other words an affine subspace, and we take it to have non-empty interior (if not, then we argue it should belong to a stratum of lower dimension.) We shall say that the fibration is \emph{adapted} to the distribution $X \in \Gamma ( TM_0 /D )$.

The statement of the theorem to follow includes a final extra assumption on $\cC$. In order to obtain a genericity statement, we shall consider control fibrations which, over the submanifold $M_0$ are actually smooth manifolds with boundary:  a control fibration  $\cC$ adapted to the affine control distribution will now be taken to be a smooth \emph{immersion} of an $m$-manifold with boundary, $i : \cC \to TM$, so the compact fibres are the intersections of the image of the manifold $\cC$ with each $X(p)$.
Thus, in our fibration diagram
\begin{equation*}
\begin{tikzcd}
\cC \arrow[r, "i"] \arrow[dr,"p"] & TM \arrow{d}{\pi} \\
 &  M 
\end{tikzcd}
\end{equation*}
the space $\cC$ is now a manifold and the map $i$ is the immersion.

We consider the usual topology on the space of such immersions, as is done in Differential Topology.

\begin{theorem}
For a dense, open subset of the space of immersions adapted to $D$, the singular set is an $m$-dimensional manifold.
\end{theorem}

\begin{proof}
We shall need the form of the Thom transversality theorem for manifolds with boundary (see for example~\cite{guill-pol}, Chapter 2.) 

The setting is again that we have the $n$-dimensional submanifold $Z \subset TM$, the zero section and a smooth map of a manifold with boundary into the vector bundle $TM/D$. By composing with the projection, we get the map $p = \pi \circ i$. 
Applying the transversality theorem we get that for a dense subset of the set of such maps, we get a transverse intersection with the zero section $Z$. The inverse image is a submanifold of $\cC$. 
Since the projection map $\pi$ is open and the fibres of $\cC$ are compact, it follows that we have a proper map. This allows us to conclude the openness of the set of fibrations giving a singular set which is an $m$-manifold, in addition to the density. The details are omitted.
\end{proof}

\section{The singular set and global control design}

Once we have established that, for a generic control system, the singular set is a manifold of the same dimension as the control (at least on some dense subset of state space), we have available the two $m$ dimensional geometric objects, the regular control distribution and the singular manifold $\Sigma^m$. The global control dynamics methodology which we have developed in a series of works (e.g.\ \cite{ek1}, \cite{ek2}, \cite{ek3}), has as its main tool the following geometric object:
\begin{definition}
A \textbf{control-transverse manifold} is a smooth submanifold $W$ of $M^n$ of complementary dimension to the control distribution, namely $(n-m)$, and which is everywhere transverse to $D$:
\[ \forall p \in W : \; T_p W \oplus D_p = T_p M . \]
\end{definition}
The local existence of such transverse manifolds is easy to establish, given the local trivializations of any vector bundle.

The importance of these geometric objects is that on them are defined, in a completely natural and invariant way, dynamics, which we call control-transverse dynamics.
Let us explain this, in the case of a control-affine system: starting with the quotient vector space construction, we then have a short exact sequence of vector bundles
\[ 0 \to D \to TM \to TM/D \to 0 . \]
Note that there is no natural selection of a complement to the vector subspace $D_p$ of $T_p M$ (there is certainly no need for a metric yet.) The control transverse submanifold $W$ provides exactly such a complement, as we saw, along $W$:
\[ 0 \to D|_W \to TM |_W \to TW \to 0 . \]
In fact, we can generalize to the notion of a \emph{control-transverse foliation}, which is taken to be a regular foliation everywhere transverse to $D$ and of complementary dimension. Again, the local existence is not difficult, but we are by no means implying that it exists globally.

Now we can lift the short exact sequence at the level of sections
\[ 0 \to \Gamma ( D ) \to \Gamma ( TM ) \to \Gamma ( TM/D ) \to 0 , \]
and, restricting to $W$, we get
\[ 0 \to \Gamma (  D|_W ) \to \Gamma (TM |_W) \to \Gamma ( TW ) \to 0 . \]
This sequence is split, so that a vector field decomposes into a section of the control distribution and a section of the tangent bundle $TW$, in other words a dynamical system (vector field) on $W$.

A first crucial remark is that the control-transverse manifolds are `soft', in the sense that, being transverse, they can be deformed quite freely. The second crucial point is that \emph{the geometry of the chose control-transverse manifold determines the dynamics!} The final point is that in the directions transverse to the chosen $W$, we have available the control directions, and we can therefore design any dynamics we wish  --for example, we can make the control-transverse manifold invariant, and asymptotically stable.

As an illustration, let us go back to a single-input linear control system, in controller-canonical form. We explained in section~\ref{lcs1} that the singular set is the $x_1$ axis. Since $D= \spanek ( e_n )$, every graph of a linear function $x_n = \sum_{i=1}^{n-1} k_i x_i $ will give a control-transverse manifold, here a vector subspace of dimension $(n-1)$ (a hyper-plane). The control-transverse dynamics are easily seen to be
\begin{align*}
 \dot x_1 & =  x_2 \\
 \dot x_2 & = x_3 \\
 \cdots & \cdots \\
 \dot x_{n-1} & = x_n = \sum_{i=1}^{n-1} k_i x_i ,
\end{align*}
which of course can be made to have arbitrary poles, depending on the choice of the coefficients $k_i$. So here different hyper-planes (geometry) give different stability properties (dynamics). Note that we get the same by considering linear feedback control $u(x) = \sum_{i=1}^n k_i x_i $ and considering the set where $u(x) =0$. The fact that the coefficient of $x_n$ is non-zero (so we get control-transverse object) makes sense since, otherwise, the last column of the feedback matrix $A + b k^t$ is zero. 

Now we also see the role of the singular set: we had better make the chosen hyper-plane transverse to the singular set as well, since, otherwise, we get a non-isolated equilibrium point! (since in that case the first column would be zero.)

The general case proceeds in a similar way. We take into account the singular set in selecting control-transverse submanifolds that are transverse to it, so we get isolated equilibrium points, and we control the overall dynamics in two stages: first, by picking the geometry of $W$ to obtain desirable control-transverse dynamics and then use the control directions to design the overall dynamics. 
The deformations of $W$ may lead to bifurcations of the control-transverse dynamics.
When we say bifurcations, of course, we mean that the setting is similar to that of the classical bifurcation theory for vector fields. A more precise definition would use the notion of a \emph{parametrized family of transverse manifolds}, $W ( \mu )$, with each manifold in the family staying transverse to $D$ -- but not necessarily to the singular set!
An example of such a bifurcation is given in~\cite{ek1}.

In fact, it is easy to verify, in many cases, the conditions for having a particular bifurcation (say the saddle-node one), since we have complete freedom in selecting the transverse $W$, locally (as we just saw in the special case of linear systems).

The full development of this theory for obtaining global controlled dynamics is yet to be completed. It seems unavoidable that we shall need to take into account the stratification of the control fibration if we are to have a truly global theory. There is in recent years a significant amount of research in other areas of mathematics and physics attempting to deal with these same issues (e.g.\ in Poisson geometry and the notion of Lefschetz fibrations \cite{androul}, \cite{akbulut}.)

%%%%%%%%%%%%%%%%%%%%%%%%%%%%%%%%%%%%%%%%%%%%%%%%%%%%%%%%%%%%%%%%%%%%%%%%%%%%%%

\medskip

\rightline{\emph{School of Mathematics, Aristotle University of Thessaloniki, Greece}}


\begin{thebibliography}{99}

\bibitem{akbulut}
S. Akbulut and C. Karakurt, Every 4-–manifold is BLF, J. G\"{o}kova Geom. Topol. GGT 2
2008 pp.83-–106.

\bibitem{androul}
I. Androulidakis and G. Skandalis, The holonomy groupoid of a singular foliation, 	J. Reine Angew. Math. 626, 2009, 1--37

\bibitem{via-sdm}
V.I. Arnold, S. Gusein-Zade and A. Varchenko, Singularities of Differentiable Maps, Vol.I, Birkh\"{a}user, 1985.

\bibitem{brockett1}
R. Brockett, Feedback invariants for nonlinear systems, Math.Syst.Th. 1976.

\bibitem{brockett2}
R. Brockett, Nonlinear systems and differential geometry, IEEE Proc. 1976.

\bibitem{coron-bk}
J.-M. Coron, Control and Nonlinearity, AMS, 2007.

% Coron mentions fibration on p.167
\bibitem{coron1}
J.-M. Coron, Links between local controllability and local continuous stabilization, in~\cite{fliess-ed}, pp.165-171.

\bibitem{fliess-ed}
M. Fliess ed., Nonlinear Control Design, Bordeaux 1992, Pergamon 1993.

\bibitem{gol-guill}
M. Golubitsky and V. Guillemin, Stable Mappings and their Singularities, Springer, 1973.

\bibitem{guill-pol}
V. Guillemin and A. Pollack, Differential Topology, Prentice Hall 1974.

\bibitem{isidori}
A. Isidori, Nonlinear Control Systems, 3rd edition, Springer, 1995.

% Jakubczyk on fvb equivalence
\bibitem{bronik1}
B. Jakubczyk, Remarks on equivalence and linearization of nonlinear systems, in~\cite{fliess-ed}, pp.143--147.

\bibitem{ek1} 
E. Kappos, Control-Transverse Dynamics and Bifurcation Behaviour, Int. Journal of Diff. Equations and Applications, \textbf{12}–-4, December 2013, pp.197-–207.

\bibitem{ek2} 
E. Kappos, Control-Transverse and Graph Dynamics and Relations to Backstepping, WSEAS Transactions on Mathematics, Vol \textbf{2}--4, October 2003, pp.337--42.

\bibitem{ek3}
E. Kappos, A Geometric Approach to Linear and Nonlinear Achievable Dynamics, Int. Journal of Pure and Applied Math., 14–-3, 2004, 397–-408.

\bibitem{khalil}
H. Khalil, Nonlinear Systems, 3rd edition, Prentice--Hall 2002.

\bibitem{krener1}
A. Krener, A generalization of Chow's theorem and the bang-bang theorem to nonlinear control problems, SIAM J. Control, 1974.

\bibitem{mayne-nato}
D.Q. Mayne and R. Brockett, Geometric Methods in System Theory NATO, Reidel 1973.

\bibitem{mormul1}
P. Mormul, Singularities of triples of vector fields on $\R^4$, Bull. Pol. Acad. Sc. \textbf{31}, 1983.

\bibitem{mormul2}
P. Mormul, Singularities of triples of vector fields on $\R^4$, the focusing stratum, Stud.Math. 1988.

\bibitem{mormul-rouss}
P. Mormul and R. Roussarie, Geometry of triples of vector fields in $\R^4$, in Pnevmatikos N. ed. Singularities and Dynamical Systems, North Holland 1985, pp.89--98.

\bibitem{sussmann-min}
H. Sussmann, Minimal realizations of nonlinear systems, in~\cite{mayne-nato}, pp.243--252.

\bibitem{zhit}
M. Zhitomirskii, Singularities and normal forms of smooth distributions, Banach Center 1995.

\end{thebibliography}
\end{document}